\documentclass[12pt]{amsart}
\usepackage{latexsym,amsmath,amssymb,epsfig,amsthm}
\usepackage{graphicx}
\usepackage{tikz}
\usetikzlibrary{calc}
\usepackage[enableskew,vcentermath]{youngtab}
\usepackage{hyperref}
\hypersetup{colorlinks=false}
\input epsf
\newtheorem{theorem}{Theorem}[section]
\newtheorem{proposition}[theorem]{Proposition}
\newtheorem{corollary}[theorem]{Corollary}

\newtheorem{example}[theorem]{Example}
\newtheorem{lemma}[theorem]{Lemma}
\newtheorem{remark}[theorem]{Remark}

\newtheorem{question}[theorem]{Question}
\newcommand{\asc}{{\rm asc}}
\newcommand{\Asc}{{\rm Asc}}

\newcommand{\des}{{\rm des}}
\newcommand{\Des}{{\rm Des}}

\newcommand{\NC}{{\rm NC}}

\newcommand{\dD}{{\mathcal D}}

\newcommand{\sS}{{\mathcal S}}
\newcommand{\tT}{{\mathcal T}}
\newcommand{\wW}{{\mathcal W}}
\newcommand{\xX}{{\mathcal X}}
\newcommand{\zZ}{{\mathcal Z}}

\newcommand{\RR}{{\mathbb R}}
\newcommand{\fS}{{\mathfrak S}}
\newcommand{\NN}{{\mathbb N}}
\newcommand{\ZZ}{{\mathbb Z}}
\newcommand{\CC}{{\mathbb C}}
\newcommand{\rzz}{{\mathrm {\mathbf z}}}
\renewcommand{\to}{\rightarrow}

\newcommand{\sm}{{\smallsetminus}}
\begin{document}
\title[Two classes of posets with real-rooted 
chain polynomials]
{Two classes of posets with real-rooted chain 
polynomials}

\author[C.A.~Athanasiadis]{Christos~A.~Athanasiadis}
\address{Department of Mathematics\\
National and Kapodistrian University of Athens\\
Panepistimioupolis\\ 15784 Athens, Greece}
\email{caath@math.uoa.gr}

\author[T.~Douvropoulos]{Theo~Douvropoulos}
\address{Department of Mathematics\\
Brandeis University\\
Waltham, MA 02453, USA}
\email{tdouvropoulos@brandeis.edu}

\author[K.~Kalampogia-Evangelinou]
{Katerina~Kalampogia-Evangelinou}
\address{Department of Mathematics\\
National and Kapodistrian University of Athens\\
Panepistimioupolis\\ 15784 Athens, Greece}
\email{kalampogia@math.uoa.gr}

\date{November 30, 2025}
\thanks{First and third authors supported by the 
Hellenic Foundation for Research and Innovation 
(H.F.R.I.) under the `2nd Call for H.F.R.I. Research 
Projects to support Faculty Members \& Researchers' 
(Project Number: HFRI-FM20-04537).}

\begin{abstract}
The coefficients of the chain polynomial of a finite 
poset enumerate chains in the poset by their number 
of elements. It has been a challenging open problem 
to determine which posets have real-rooted chain 
polynomials. Two new classes of posets, namely those 
of all rank-selected subposets of Cohen-Macaulay 
simplicial posets and all noncrossing partition 
lattices associated to finite Coxeter groups, are 
shown to have this property. The first result 
generalizes one of Brenti and Welker. As a special 
case, the descent enumerator of permutations of the 
set $\{1, 2,\dots,n\}$ which have ascents at 
specified positions is shown to be real-rooted, hence 
log-concave and unimodal, and a good estimate for the 
location of the peak is deduced.        
\end{abstract}

\maketitle

\section{Introduction}
\label{sec:intro}

The coefficients of the chain polynomial $f_P(x)$
of a finite poset (partially ordered set) $P$ 
enumerate chains in $P$ by their number of elements. 
From a face enumeration point of view, $f_P(x)$ is 
the $f$-polynomial of a flag simplicial complex of 
special type, namely the order complex $\Delta(P)$ 
(see~\cite{StaCCA} for basic definitions and 
terminology on simplicial complexes and their face 
enumeration). Thus, 
\begin{equation}
\label{eq:f-Delta(Q)-def}
f_P(x) = f(\Delta(P),x) = \sum_{i=0}^n f_{i-1}
          (\Delta(P)) x^i,
\end{equation}
where $f_{i-1}(\Delta(P))$ is the number of 
$i$-element chains in $P$ (which are precisely the
$(i-1)$-dimensional faces of $\Delta(P)$) and $n$
is the largest size of such a chain (which is one 
more than the dimension of $\Delta(P)$). For some 
purposes, one may focus on the corresponding
$h$-polynomial
\begin{equation}
\label{eq:h-Delta(Q)-def}
h(\Delta(P),x) = \sum_{i=0}^n f_{i-1}(\Delta(P)) 
                 x^i (1-x)^{n-i}
\end{equation}
instead. For instance, $f_P(x)$ has only real 
roots if and only if so does $h(\Delta(P),x)$. 
The following general question was posed 
in~\cite{AK23}.
\begin{question} \label{que:main} 
{\rm (\cite[Question~1.1]{AK23})}
For which finite posets $P$ does the chain polynomial 
$f_P(x)$ have only real roots?  
\end{question}

There are several motivations behind this question.
For finite distributive lattices, it is known to be 
equivalent to the poset conjecture for natural 
labelings, posed in the seventies by 
Neggers~\cite{Ne78} (see also 
\cite[Conjecture~1]{Sta89}) and finally disproved by 
Stembridge~\cite{Ste07}, after counterexamples to a 
more general conjecture were found by 
Br\"and\'en~\cite{Bra04}. It is open for face 
posets of convex polytopes \cite[Question~1]{BW08}
and more general regular cell complexes 
\cite[Section~5]{AK23}, in which case the chain 
polynomial coincides with the $f$-polynomial of 
the barycentric subdivision of the complex, and 
is conjectured to hold for all geometric 
lattices \cite[Conjecture~1.2]{AK23}. Among other 
positive results, an affirmative answer to 
Question~\ref{que:main} has been given for 
simplicial posets with nonnegative 
$h$-vector~\cite{BW08} and cubical posets with 
nonnegative cubical $h$-vector~\cite{Ath21} (in 
particular, for face lattices of simplicial and 
cubical convex polytopes), for partition and 
subspace lattices~\cite{AK23} and for posets 
which do not contain the disjoint union of a 
three-element chain and a one-element chain as an 
induced subposet~\cite{Sta98}. An overview of 
positive results will appear in~\cite{Ka24+}.

This paper contributes an affirmative answer to 
Question~\ref{que:main} for two other broad classes 
of posets, namely those of all rank-selected 
subposets of Cohen--Macaulay (over some field) 
simplicial posets (more generally, of simplicial 
posets with nonnegative $h$-vector) and all 
noncrossing partition lattices associated to 
finite Coxeter groups. A special case of the first 
class of independent interest comes from 
considering rank-selected subposets of Boolean 
lattices. Given $T \subseteq [n-1] := 
\{1, 2,\dots,n-1\}$ we set 
\begin{equation} \label{eq:AnT-def}
A^T_n(x) = \sum_{w \in \fS_n : \, 
           \Des(w) \subseteq T} x^{\des(w)},
\end{equation}
where $\fS_n$ is the symmetric group of 
permutations of the set $[n]$ and for $w \in \fS_n$, 
$\Des(w) = \{ i \in [n-1]: w(i) > w(i+1)\}$ and 
$\des(w)$ is the set and the number of descents of
$w$, respectively. Thus, $A^T_n(x)$ is the descent 
enumerator for permutations in $\fS_n$ which have 
ascents at specified positions (those in $[n-1] \sm
T$); it reduces to the classical Eulerian polynomial 
$A_n(x)$ \cite[Section~1.4]{StaEC1} for $T = [n-1]$. 
A $q$-analogue of $A^T_n(x)$ was studied 
in~\cite{CGSW07} in the special case $T = [r] 
\subseteq [n-1]$ (see also Example~\ref{ex:T=[r]}). 
To the best of our knowledge, except for 
the important special case of Eulerian polynomials, 
the unimodality, log-concavity and real-rootedness 
of $A^T_n(x)$ have not been considered before.

The following statement is the first main result 
of this paper (basic definitions and terminology 
on posets can be found in Sections~\ref{sec:pre}, 
\ref{sec:simplicial}, \ref{sec:noncrossing}). 
Part (b) generalizes the main result of~\cite{BW08}, 
which corresponds to the special case $T = [n]$. 
Throughout this paper, we denote by $\hat{P}$ the 
poset obtained from $P$ by adding a maximum element 
and set $A^T_n(x) := A^{T \cap [n-1]}_n (x)$ for 
$T \subseteq \NN$. 
\begin{theorem} \label{thm:mainA} 
Let $n$ be a positive integer.
\begin{itemize}
\itemsep=0pt
\item[(a)]
The polynomial $A^T_n(x)$ has only real roots for 
every $T \subseteq [n-1]$.

\item[(b)]
Let $P$ be a Cohen--Macaulay simplicial poset of 
rank $n$. Then, every rank-selected subposet 
$\hat{P}_T$ of $\hat{P}$ has a real-rooted chain 
polynomial. Moreover, $h(\Delta(\hat{P}_T),x)$ is 
interlaced by $A_n^T(x)$ for every $T \subseteq [n]$.

\end{itemize}
\end{theorem}

Noncrossing partition lattices associated to Coxeter
groups are central objects of study in Coxeter-Catalan
combinatorics; see \cite[Chapter~2]{Arm09} for an 
overview. The enumeration of chains in these posets
has been a very popular topic \cite{AR04, CD22, Ed80, 
JV15, Kim11, Rea08, Rei97}. Let us denote by $\NC_W$ 
the noncrossing partition lattice associated to $W$. 
The second main result of this paper is as follows.
\begin{theorem} \label{thm:mainB} 
Let $W$ be a finite Coxeter group.
\begin{itemize}
\itemsep=0pt
\item[(a)]
The noncrossing partition lattice $\NC_W$ has a 
real-rooted chain polynomial.

\item[(b)]
The polynomial $h(\Delta(\NC_W), x)$ has a 
nonnegative real-rooted symmetric decomposition with 
respect to $r_W-1$ for every irreducible finite 
Coxeter group $W$, where $r_W$ is the rank of $W$. 
In particular, $h(\Delta(\NC_W),x)$ is unimodal, 
with a peak at position $\lfloor r_W/2 \rfloor$. 

\end{itemize}
\end{theorem}

Question~\ref{que:main} cannot have an affirmative 
answer for all Cohen--Macaulay posets since, as 
already explained, it fails for finite distributive 
lattices. However, and since the proper parts of 
face lattices of polytopes, 
geometric lattices, noncrossing partition lattices 
of types $A$ and $B$ \cite{KK13} and rank-selected 
subposets of Boolean lattices are doubly 
Cohen--Macaulay (see \cite[Section~III.3]{StaCCA} 
for information about doubly Cohen--Macaulay posets), 
it seems reasonable to pose the following question.
\begin{question} \label{que:2CM} 
Does the chain polynomial of any doubly 
Cohen--Macaulay lattice (or even doubly 
Cohen--Macaulay poset) have only real roots?
\end{question}

This paper is organized as follows.
Section~\ref{sec:pre} reviews definitions and tools 
from the theory of real-rooted polynomials (and 
especially the theory of interlacing) and the
enumerative combinatorics of posets which are 
essential in understanding the main results and 
their proofs. The proof of Theorem~\ref{thm:mainA}
splits in two sections. Section~\ref{sec:perms} 
proves that $A^T_n(x)$ is real-rooted (see 
Theorem~\ref{thm:perms}), hence unimodal, gives a 
good estimate for the location of the peak and 
discusses some interesting special cases. 
Section~\ref{sec:simplicial} proves part (b) of 
the theorem by combining Theorem~\ref{thm:perms}
with an exercise from \cite{StaCCA} (see 
Lemma~\ref{lem:StaCCA}) and, as an application, 
generalizes part (a) in the setting of colored 
permutations. Part (a) of Theorem~\ref{thm:mainB} 
is proven in Section~\ref{sec:noncrossing} in two 
different ways. The first proof does not assume 
the classification of finite Coxeter groups. The 
second proof is based on explicit combinatorial 
interpretations (which are of independent 
interest), as descent enumerators of certain 
families of words, of the $h$-polynomials of the 
order complexes $\Delta(\NC_W)$ for the irreducible 
finite Coxeter groups $W$ of classical types (see 
Proposition~\ref{prop:noncrossing}) and on computer 
computations for the exceptional groups. These 
combinatorial interpretations are extracted from 
the known explicit formulas for the entries of the 
flag $f$-vectors of noncrossing partition lattices 
\cite{AR04, Ed80, Rei97}, the case of groups of 
type $D$ being the trickiest. Part (b) of 
Theorem~\ref{thm:mainB} follows from 
Proposition~\ref{prop:noncrossing} by an 
application of a result of Jochemko~\cite{Jo21} 
about Veronese operators on formal power series.

\section{Preliminaries}
\label{sec:pre}

This section reviews basic concepts and tools from 
the theory of real-rooted polynomials and the 
enumerative combinatorics of posets (the theory of 
rank selection, in particular) which will be 
essential in the following three sections. Standard 
references for these topics are \cite{Bra15, Fi06,
Sta89, StaCCA, StaEC1}.

\subsection{Polynomials} 
A polynomial $p(x) = h_0 + h_1 x + \cdots + h_n x^n 
\in \RR[x]$ is called
\begin{itemize}
\item[$\bullet$] 
  \emph{symmetric}, with center of symmetry $n/2$, if 
	$h_i = h_{n-i}$ for all $0 \le i \le n$,
\item[$\bullet$] 
  \emph{unimodal}, with a peak at position $k$, if 
	$h_0 \le h_1 \le \cdots \le h_k \ge h_{k+1} \ge 
	\cdots \ge h_n$,
\item[$\bullet$] 
  \emph{log-concave}, if $h^2_i \ge h_{i-1}h_{i+1}$ 
	for $1 \le i \le n-1$,
\item[$\bullet$] 
  \emph{real-rooted}, if every root of $p(x)$ is 
	real, or $p(x) \equiv 0$.
\end{itemize}
Every real-rooted polynomial with nonnegative 
coefficients is log-concave and unimodal; see
\cite{Bra15, Sta89} for more information about these 
concepts. 

A real-rooted polynomial $p(x)$, with 
roots $\alpha_1 \ge \alpha_2 \ge \cdots$, is 
said to \emph{interlace} a real-rooted polynomial 
$q(x)$, with roots $\beta_1 \ge \beta_2 \ge \cdots$, 
if
\[ \cdots \le \alpha_2 \le \beta_2 \le \alpha_1 \le
   \beta_1. \]
We then write $p(x) \preceq q(x)$. By convention, 
the zero polynomial interlaces and is interlaced by 
every real-rooted polynomial and nonzero constant 
polynomials interlace all polynomials 
of degree at most one. A sequence 
$(p_0(x), p_1(x),\dots,p_m(x))$ of real-rooted 
polynomials is called \emph{interlacing} if 
$p_i(x) \preceq p_j(x)$ for $0 \le i < j \le 
m$. The following statement lists well known 
properties of interlacing sequences; see, for 
instance, \cite[Section~7.8]{Bra15} 
\cite[Chapter~3]{Fi06}.
\begin{lemma} \label{lem:interlace-rec} 
Let $(p_0(x), p_1(x),\dots,p_m(x))$ be an 
interlacing sequence of real-rooted polynomials 
with positive leading coefficients.
\begin{itemize}
\itemsep=0pt
\item[(a)]
Every nonnegative linear combination $p(x)$ of 
$p_0(x), p_1(x),\dots,p_m(x)$ is real-rooted. 
Moreover, $p_0(x) \preceq p(x) \preceq p_m(x)$. 

\item[(b)]
The sequence $(q_0(x), q_1(x),\dots,q_{m+1}(x))$ 
of partial sums
\[ q_k(x) = \sum_{i=k}^m p_i(x) \]
for $k \in \{0, 1,\dots,m+1\}$ is also interlacing.

\item[(c)]
The sequence $(t_0(x), t_1(x),\dots,t_{m+1}(x))$ 
defined by
\[ t_k(x) = x \sum_{i=0}^{k-1} p_i(x) + 
             \sum_{i=k}^m p_i(x) \]
for $k \in \{0, 1,\dots,m+1\}$ is also interlacing.

\end{itemize}
\end{lemma}

Given a polynomial $p(x) \in \RR[x]$ of degree 
at most $n$, there exist unique symmetric 
polynomials $a(x), b(x) \in \RR[x]$ with centers 
of symmetry $n/2$ and $(n-1)/2$, respectively, 
such that $p(x) = a(x) + xb(x)$. This expression 
is known as the \emph{symmetric decomposition} (or 
\emph{Stapledon decomposition}) of $p(x)$ with 
respect to $n$. Then, $p(x)$ is said to have a 
\emph{nonnegative} (respectively, \emph{unimodal} 
or \emph{real-rooted}) \emph{symmetric 
decomposition} with respect to $n$ if $a(x)$ and 
$b(x)$ have nonnegative coefficients (respectively, 
are unimodal or real-rooted); see \cite{AT21, BS21} 
for more information about these concepts. Every
polynomial which has a nonnegative unimodal 
symmetric decomposition with respect to $n$ is 
unimodal, with a peak at position $\lceil n/2 
\rceil$.

\subsection{Poset combinatorics}
Our notation and terminology generally follows that 
of \cite[Chapter~3]{StaEC1}. Let $P$ be a finite 
graded poset of rank $n$, having a minimum element 
$\hat{0}$ and rank function $\rho: P \to 
\{0, 1,\dots,n\}$, and let $\hat{P}$ be the poset 
obtained from $P$ by adding a maximum element 
$\hat{1}$. Given $T \subseteq [n]$, the 
$T$-rank-selected subposet of $\hat{P}$ is defined 
as 
\[ \hat{P}_T = \{ y \in P : \rho(y) \in T \} \cup 
               \{\hat{0}, \hat{1}\}. \]
We denote by $\alpha_{\hat{P}}(T)$ the number of 
maximal chains of $\hat{P}_T$ and set 
\begin{equation} \label{eq:ab} 
\beta_{\hat{P}}(T) = \sum_{S \subseteq T} 
   (-1)^{|T \sm S|} \alpha_{\hat{P}}(S) 
\end{equation}
for $T \subseteq [n]$. Equivalently, we have 
\begin{equation} \label{eq:ba} 
\alpha_{\hat{P}}(T) = \sum_{S \subseteq T} 
                      \beta_{\hat{P}}(S) 
\end{equation}
for $T \subseteq [n]$. The collections of numbers 
$(\alpha_{\hat{P}}(T))_{T \subseteq [n]}$ and 
$(\beta_{\hat{P}}(T))_{T \subseteq [n]}$ are the 
\emph{flag $f$-vector} and the 
\emph{flag $h$-vector} of $\hat{P}$, respectively.

The \emph{order complex} of a finite poset $Q$ is 
defined as the simplicial complex $\Delta(Q)$ which
consists of all chains in $Q$. The $f$-polynomial 
and the $h$-polynomial of $\Delta(Q)$ are defined by 
Equations~(\ref{eq:f-Delta(Q)-def}) 
and~(\ref{eq:h-Delta(Q)-def}), respectively, when 
$P$ is replaced by $Q$. Since the $h$-polynomial 
is unaffected when maximum or minimum elements are 
removed from $Q$, we have the equivalent expressions
\begin{equation}
\label{eq:fP_T}
f(\Delta(\hat{P}_T \sm \{\hat{0}, \hat{1}\}), x) = 
\sum_{S \subseteq T} \alpha_{\hat{P}}(S) x^{|S|} = 
\sum_{S \subseteq T} \beta_{\hat{P}}(S) x^{|S|} 
(1+x)^{|T \sm S|} 
\end{equation}
and
\begin{equation}
\label{eq:hP_T}
h(\Delta(\hat{P}_T), x) = \sum_{S \subseteq T} 
  \alpha_{\hat{P}}(S) x^{|S|} (1-x)^{|T \sm S|} = 
	\sum_{S \subseteq T} \beta_{\hat{P}}(S) x^{|S|}
\end{equation}
for every $T \subseteq [n]$, where the second 
equality in each case is a consequence of 
Equation~(\ref{eq:ba}).

\begin{example} \label{ex:boolean} \rm
Let $\hat{P}$ be the Boolean lattice $B_n$ of 
subsets of $[n]$, partially ordered by inclusion.
Then, $\beta_{\hat{P}}(S)$ is equal to the number 
of permutations $w \in \fS_n$ with $\Des(w) = S$ 
for every $S \subseteq [n-1]$ 
\cite[Corollary~3.13.2]{StaEC1} and 
Equation~(\ref{eq:hP_T}) yields that 
\[ h(\Delta((B_n)_T),x) = 
   \sum_{w \in \fS_n : \, \Des(w) \subseteq T} 
	 x^{\des(w)} = A^T_n(x). \]
We note that, by definition of $A^T_n(x)$ and a 
standard argument, we have $A^T_n(x) = 
A^{n-T}_n(x)$ for every $T \subseteq [n-1]$, 
where $n-T := \{ n - a : a \in T\}$. 
\qed
\end{example}

The \emph{zeta polynomial} $\zZ(P,x)$ is another 
important enumerative invariant of a finite poset
$P$ \cite[Section~3.12]{StaEC1}. For the ease of 
notation, we define it here by letting $\zZ(P,k)$
be the number of multichains $p_1 \le p_2 \le \cdots
\le p_k$ of length $k-1$ (rather than $k-2$) of 
elements of $P$, where $\zZ(P,0) := 1$. A 
comparison of \cite[Proposition~3.12.1~(a)]{StaEC1}
with \cite[Theorem~II.1.4]{StaCCA} then shows that
\begin{equation} \label{eq:zeta}
\sum_{k \ge 0} \zZ(P,k) x^k = \frac{h(\Delta(P),x)}
{(1-x)^n},
\end{equation}
where $n$ is the largest cardinality of a chain in
$P$. The following lemma states that the product 
$P \times Q$ of posets has a real-rooted chain 
polynomial, provided that so do $P$ and $Q$ and 
that $\Delta(P)$ and $\Delta(Q)$ have nonnegative 
$h$-vectors; it complements some of the results of 
\cite[Section~5]{AK23}.
\begin{lemma} \label{lem:product}
Let $P, Q$ be finite posets. If $h(\Delta(P),x)$ 
and $h(\Delta(Q),x)$ have nonnegative coefficients
and only real roots, then so does 
$h(\Delta(P \times Q),x)$.
\end{lemma} 

\begin{proof}
Let $m$ and $n$ be the largest cardinality of a 
chain in $P$ and $Q$, respectively. Clearly, we 
have $\zZ(P \times Q, k) = \zZ(P, k) \zZ(Q, k)$
for every $k$. Hence, by Equation~(\ref{eq:zeta}), 
\[ \sum_{k \ge 0} \zZ(P,k) \zZ(Q,k) x^k = 
   \frac{h(\Delta(P \times Q),x)}{(1-x)^{m+n-1}}
	 \]
and the proof follows by an application of 
\cite[Theorem~0.2]{Wa92}.
\end{proof}

\section{Permutations with restricted descent set}
\label{sec:perms}

This section proves that the polynomials $A^T_n(x)$
are real-rooted, as claimed in part (a) of 
Theorem~\ref{thm:mainA}, and in particular unimodal,
locates their peak and discusses some interesting 
special cases and formulas. The applications of the
real-rootedness of $A^T_n(x)$ discussed here have 
a probabilistic flavor; see \cite[Section~7.2]{Bra15}
\cite{Pi97} for overviews of this topic. For a 
probabilistic approach to the theory of descents 
in permutations, we recommend 
\cite[Section~5]{BDF10}.

Crucial to the proof will be the polynomials
\begin{equation} \label{eq:pnkT-def}
p^T_{n,k}(x) = \sum_{w \in \fS_{n+1,k+1} : \, 
               \Des(w) \subseteq T} x^{\des(w)},
\end{equation}
where $T \subseteq [n]$, $k \in \{0, 1,\dots,n\}$
and $\fS_{n+1,k+1}$ is the set of permutations 
$w \in \fS_{n+1}$ such that $w(1) = k+1$. We note 
that $p^T_{n,0}(x) = A^{T-1}_n(x)$ and 
\[ p^T_{n,n}(x) = \begin{cases}
   x A^{T-1}_n(x), & \text{if $1 \in T$} \\
   0, & \text{if $1 \not\in T$} \end{cases} \]
for $T \subseteq [n]$, where $T - 1 := \{ a-1: a 
\in T \}$ and, as mentioned in 
Section~\ref{sec:intro}, $A^T_n(x) := 
A^{T \cap [n-1]}_n(x)$ for $T \subseteq \NN$. 
We set $p^T_{n,k}(x) = p_{n,k}(x)$ when $T = 
[n]$; these polynomials appeared in \cite{BW08} 
\cite[Section~2.2]{CGSW07} and have 
been studied intensely since then; see, for instance, 
\cite[Section~2]{Ath22+} \cite[Section~3]{BJM19} 
\cite[Example~7.8.8]{Bra15} and the references 
given there. They can also be defined by the 
formula \cite[Equation~(4)]{BW08}
\begin{equation} \label{eq:pnk-gen}
\sum_{m \ge 0} m^k (1+m)^{n-k} x^m = 
   \frac{p_{n,k}(x)}{(1-x)^{n+1}}.
\end{equation}
The polynomials $p_{n,k}(x)$ are real-rooted and 
$(p_{n,k}(x))_{0 \le k \le n}$ is an interlacing 
sequence for every $n \in \NN$; see, for instance,
\cite[Example~7.8.8]{Bra15}. This fact is 
generalized by the main result of this section.

\begin{theorem} \label{thm:perms} 
For all $n \in \NN$ and $T \subseteq [n]$, 
\begin{equation} \label{eq:pnkT-seq}
(p^T_{n,0}(x), p^T_{n,1}(x),\dots,p^T_{n,n}(x))
\end{equation}  
is an interlacing sequence of real-rooted 
polynomials.

In particular, $A^{T-1}_n(x)$ is real-rooted and 
it interlaces $A^T_{n+1}(x)$ for all positive 
integers $n$ and $T \subseteq [n]$.
\end{theorem}

\begin{proof}
We proceed by induction on $n$, the result being 
trivial for $n=0$. Suppose that $n \ge 1$ and 
that~(\ref{eq:pnkT-seq}) is an interlacing
sequence of real-rooted polynomials when $n$ is
replaced by $n-1$. It is straightforward to verify 
from the defining equation~(\ref{eq:pnkT-def}) 
that 
\[ p^T_{n,k}(x) = \sum_{i=k}^{n-1} p^{T-1}_{n-1,i}
   (x) \]
for $k \in \{0, 1,\dots,n\}$, if $1 \not\in T$, 
and
\[ p^T_{n,k}(x) = x \sum_{i=0}^{k-1} 
   p^{T-1}_{n-1,i}(x) + 
	 \sum_{i=k}^{n-1} p^{T-1}_{n-1,i}(x) \]
for $k \in \{0, 1,\dots,n\}$, if $1 \in T$. This 
recurrence generalizes that of the special case $T
= [n]$; see \cite[Example~7.8.8]{Bra15}. An 
application of Lemma~\ref{lem:interlace-rec} shows 
that, in either case, (\ref{eq:pnkT-seq}) is an 
interlacing sequence of real-rooted polynomials as 
well. This completes the inductive step. The last 
statement follows from part (a) of 
Lemma~\ref{lem:interlace-rec} since $p^T_{n,0}(x) 
= A^{T-1}_n(x)$ and $\sum_{k=0}^n p^T_{n,k}(x) = 
A^T_{n+1}(x)$.
\end{proof}

We recall that a polynomial $p(x) = \sum_{k \ge 0} 
h_k x^k \in \RR[x]$ with nonnegative and unimodal 
coefficients is said to have a \emph{mode} $m$ if 
there exists a unique $m \in \frac{1}{2} \ZZ$ such 
that either $h_m = \max_k h_k$, or $h_{m \pm 1/2} 
= \max_k h_k$.  
\begin{corollary} \label{cor:mode} 
The polynomial $A^T_n(x)$ is unimodal and log-concave 
for every $T \subseteq [n-1]$. Moreover, $A^T_n(x)$ 
has a mode $m_n(T)$ such that $\lfloor \mu_n(T) 
\rfloor \le m_n(T) \le \lceil \mu_n(T) \rceil$ for
\begin{equation} \label{eq:m_n(T)-def}
\mu_n(T) := r - \sum_{i=1}^r {c_i + c_{i+1} \choose
   c_i}^{-1},
\end{equation}  
where $T = \{a_1, a_2,\dots,a_r\}$ with $1 \le a_1 <
\cdots < a_r < n$ and $c_i = a_i - a_{i-1}$ for $i 
\in [r+1]$, with $a_0 := 0$ and $a_{r+1} := n$. In 
particular,
\begin{equation} \label{eq:A^Tn-half-increasing}
  h_0(T) \le h_1(T) \le \cdots \le 
  h_{\lfloor r/2 \rfloor}(T),
\end{equation}  
if $A^T_n(x) = \sum_{i=0}^r h_i(T) x^i$.
\end{corollary}

\begin{proof}
Since $A^T_n(x)$ is real-rooted and has nonnegative 
coefficients, by a result of 
Darroch~\cite[Theorem~4]{Da64} (see also 
\cite[Theorem~2.2]{Bra15} \cite[p.~284]{Pi97}) we 
only need to verify that the right-hand side of 
Equation~(\ref{eq:m_n(T)-def}) is equal to the 
expected number of descents when a permutation 
$w \in \fS_n$ with $\Des(w) \subseteq T$ is selected 
uniformly at random. This holds because the 
probability that $a_i \in \Des(w)$ for such $w$ is 
easily computed to be
$1 - {c_i + c_{i+1} \choose c_i}^{-1}$. The last 
statement follows since $\mu_n(T) \ge r/2$.
\end{proof}

\begin{remark} \label{rem:gessel} \rm
The fact (see Example~\ref{ex:boolean}) that $A^T_n(x) 
= h(\Delta((B_n)_T),x)$, combined with 
Equation~(\ref{eq:hP_T}), yields the explicit formula 
\[ A^T_n(x) = \sum_{S \subseteq T} \alpha_{B_n}(S) 
              \, x^{|S|} (1-x)^{|T \sm S|}, \]
where $\alpha_{B_n}(S)$ is a multinomial coefficient.
We thank Ira Gessel~\cite{Ge23} for pointing out that 
this is equivalent to the determinantal formula $x^r 
A^T_n (1/x) = n! \det 
(\theta_{ij}(x))_{0 \le i, j \le r}$, where
\[ \theta_{ij}(x) = \begin{cases}
   0, & \text{if $i > j+1$} \\
   1, & \text{if $i = j+1$} \\
	 \displaystyle \frac{(1-x)^{j-i}}{(a_{j+1}-a_i)!}, 
	 & \text{if $i \le j$} \end{cases} \]
and $T = \{a_1, a_2,\dots,a_r\}$ is as in 
Corollary~\ref{cor:mode}, and for suggesting a 
direct combinatorial proof.
\qed
\end{remark}

\begin{example} \label{ex:T=[r]} \rm
(a) For $T = [r] \subseteq [n-1]$, the polynomial 
$A^T_n(x)$ is the descent enumerator for 
permutations $w \in \fS_n$ which have ascents in 
the last $n-r-1$ positions. A $q$-analogue of 
$A^T_n(x)$ in this case was studied 
in~\cite{CGSW07} (although the unimodality of 
$A^T_n(x)$ was not addressed there). 
Theorem~\ref{thm:perms} and 
Corollary~\ref{cor:mode} imply that $A^T_n(x)$ is
a real-rooted, hence unimodal, polynomial of 
degree $r$ and that it has a mode $m$ 
such that $\lfloor r/2 \rfloor \le m \le \lceil
(r+1)/2 \rceil$. Since $(B_n)_T$, with its 
maximum element removed, is a simplicial poset 
in this case, the real-rootedness of $A^T_n(x)$ 
already follows from the main result 
of~\cite{BW08}. Setting $q=1$ in the formula of 
\cite[Theorem~2.10]{CGSW07} gives that
\[ \sum_{m \ge 0} \sum_{i=0}^r 
   {n-r+i-1 \choose i} m^i (m+1)^{r-i} = 
	 \frac{A^T_n(x)}{(1-x)^{r+1}} \]
or equivalently, by Equation~(\ref{eq:pnk-gen}), 
that 
\[ A^T_n(x) = \sum_{i=0}^r {n-r+i-1 \choose i} 
   p_{r,i}(x). \]
In particular, $A^T_n(x)$ is interlaced by the 
Eulerian polynomial $p_{r,0}(x) = A_r(x)$.

(b) More generally, for $T = \{s+1, s+2,\dots,s+r\} 
\subseteq [n-1]$, the polynomial $A^T_n(x)$ is the 
descent enumerator for permutations $w \in \fS_n$ 
which have ascents in the first $s$ and the last 
$n-r-s-1$ positions. According to 
Theorem~\ref{thm:perms} and 
Corollary~\ref{cor:mode}, $A^T_n(x)$ is a 
real-rooted, hence unimodal, polynomial of degree 
$r$ which has a mode $m$ such that $\lfloor 
(r-1)/2 \rfloor \le m \le \lceil (r+1)/2 \rceil$. 
\qed
\end{example}

\begin{example} \label{ex:even-descents} \rm
Let $T = \{2, 4,\dots,2n-2\}$, so that $A^T_{2n}(x)$ 
is the descent enumerator for permutations $w \in 
\fS_{2n}$ which have an ascent in every odd 
position. By Theorem~\ref{thm:perms} and 
Corollary~\ref{cor:mode}, $A^T_{2n}(x)$ is a 
real-rooted, hence unimodal, polynomial of degree 
$n-1$ which has a mode $m_n$ such that $\lfloor 
5(n-1)/6 \rfloor \le m_n \le \lceil 5(n-1)/6 
\rceil$. 

Let us choose a permutation $w \in \fS_{2n}$ with
$\Des(w) \subseteq T = \{2, 4,\dots,2n-2\}$ 
uniformly at random and let $X_n(w) = \des(w)$ for
such $w \in \fS_{2n}$. One may compute the variance 
of the random variable $X_n$ as $\sigma^2_n = (19n-13)
/180$ for $n \ge 2$. As a consequence of 
Corollary~\ref{cor:mode}, $X_n$ has mean $\mu_n = 
5(n-1)/6$. Given the real-rootedness of $A^T_{2n}
(x)$, a theorem of Bender~\cite{Be73} (see also 
\cite[Theorem~2.1]{Bra15} \cite[p.~286]{Pi97}) 
implies that $(X_n - \mu_n) / \sigma_n$ converges 
to the standard normal distribution as 
$n \to \infty$.
\qed
\end{example}

We conclude this section with the following 
question. Part (b) provided a lot of the 
motivation behind this paper; it is an open problem
\cite[Question~7.2]{AT21} to decide whether the
inequalities which appear there hold for the 
$h$-vectors of all $(r-1)$-dimensional doubly 
Cohen--Macaulay simplicial complexes. An 
affirmative answer to
part (a) would imply the (weaker) top-heavy 
inequalities $h_i(T) \le h_{r-i}(T)$ for $0 \le 
i \le \lfloor r/2 \rfloor$; we refer the reader 
to \cite{Sw06} for this implication and for the 
concept of a convex ear decomposition. 
\begin{question} \label{que:convex-ear} 
Let $A^T_n(x) = \sum_{i=0}^r h_i(T) x^i$, where 
$T \subseteq [n-1]$ has size $r$.
\begin{itemize}
\itemsep=0pt
\item[(a)]
Does the order complex of the rank-selected 
subposet $(B_n)_T$ of the Boolean lattice $B_n$ 
(with its minimum and maximum elements removed) 
have a convex ear decomposition? 

\item[(b)]
Do the inequalities
\[ \frac{h_0(T)}{h_r(T)} \le 
   \frac{h_1(T)}{h_{r-1}(T)} \le \cdots \le 
   \frac{h_r(T)}{h_0(T)} \]
hold?

\end{itemize}
\end{question}

\section{Rank-selected subposets of simplicial 
         posets}
\label{sec:simplicial}

This section proves part (b) of 
Theorem~\ref{thm:mainA} and gives an application.
We recall that a finite poset $P$ with a minimum
element $\hat{0}$ is said to be \emph{simplicial}
\cite{Sta91} \cite[Section~III.6]{StaCCA} if the
interval $[\hat{0},y]$ is isomorphic to a Boolean 
lattice for every $y \in P$.  
The enumerative invariant of a graded simplicial 
poset $P$ of rank $n$ which will be essential to
the proof is the \emph{$h$-polynomial} of $P$. 
This was defined by Stanley~\cite{Sta91} as 
\[ h(P, x) = \sum_{i=0}^n f_{i-1} (P) \, x^i 
             (1-x)^{n-i}, \]
where $f_{i-1} (P)$ is the number of elements of 
$P$ of rank $i$. We then have
\begin{equation} \label{eq:h-to-f}
f_{j-1}(P) = \sum_{i=0}^j {n-i \choose j-i} h_i(P)
\end{equation}
for $j \in \{0, 1,\dots,n\}$ and $h(P, x) = 
h(\Delta,x)$, if $P$ is the face poset of an 
$(n-1)$-dimensional simplicial complex $\Delta$. 
Stanley~\cite{Sta91} (see also 
\cite[Section~III.6]{StaCCA}) showed that $h(P,x)$
has nonnegative coefficients for every 
Cohen--Macaulay simplicial poset $P$.

Another essential ingredient for the proof of 
Theorem~\ref{thm:mainA} is the following statement
(an exercise from~\cite{StaCCA}), which expresses 
the flag $h$-vector of a graded simplicial poset 
in terms of its $h$-vector. We provide a proof for 
the convenience of the reader.
\begin{lemma} \label{lem:StaCCA} 
{\rm (\cite[Exercise~III.15]{StaCCA})}
Let $P$ be a graded simplicial poset of rank $n$. 
Then, 
\[ \beta_{\hat{P}} (S) = \sum_{k=0}^n h_k(P) \
   \# \{ w \in \fS_{n+1} : \, w(n+1) = k+1, \Des(w)
	 = [n+1] \sm S \} \]
for every $S \subseteq [n]$. 
\end{lemma}

\begin{proof}
Let $\Asc(w) := [n] \sm \Des(w)$ be the set of 
ascents of a permutation $w \in \fS_{n+1}$. We need 
to show that 
\[ \beta_{\hat{P}} (S) = \sum_{k=0}^n h_k(P) \
   \# \{ w \in \fS_{n+1} : \, w(n+1) = k+1, \Asc(w)
	 = S \} \]
for every $S \subseteq [n]$ or, equivalently, that 
\[ \alpha_{\hat{P}} (T) = \sum_{k=0}^n h_k(P) \
   \# \{ w \in \fS_{n+1} : \, w(n+1) = k+1, \Asc(w)
	 \subseteq T \} \]
for every $T \subseteq [n]$. Let us write $T = 
\{a_1, a_2,\dots,a_r\} \subseteq [n]$, with $1 \le 
a_1 < \cdots < a_r \le n$. There are $f_{a_r-1}(P)$
elements of rank $a_r$ in $P$ and ${a_r \choose a_1,
a_2-a_1,\dots,a_r-a_{r-1}}$ chains of elements of 
ranks $a_1, a_2,\dots,a_{r-1}$ in any Boolean 
lattice of rank $a_r$. Given this and 
Equation~(\ref{eq:h-to-f}), we find that
\begin{align*} 
\alpha_{\hat{P}} (T) & = f_{a_r-1}(P) 
   {a_r \choose a_1, a_2-a_1,\dots,a_r-a_{r-1}} \\
	 & = \sum_{k=0}^n {n-k \choose a_r-k} h_k(P) 
	 {a_r \choose a_1, a_2-a_1,\dots,a_r-a_{r-1}}.
\end{align*}
Thus, it suffices to verify that ${n-k \choose a_r-k} 
{a_r \choose a_1, a_2-a_1,\dots,a_r-a_{r-1}}$ is 
equal to the number of permutations $w \in \fS_{n+1}$
such that $w(n+1) = k+1$ and $\Asc(w) \subseteq T$,
a task which can safely be left to the reader.
\end{proof}

\begin{proof}[Proof of Theorem~\ref{thm:mainA}]
Given Theorem~\ref{thm:perms}, we only need 
to show part (b). By Lemma~\ref{lem:StaCCA} we 
have
\begin{align*}
\beta_{\hat{P}} (S) & = \sum_{k=0}^n h_k(P) \
\# \{ w \in \fS_{n+1} : \, w(n+1) = k+1, \Asc(w)
= S \} \\ & = \sum_{k=0}^n h_k(P) \
\# \{ w \in \fS_{n+1} : \, w(1) = k+1, \Des(w)
	 = n+1 - S \}. 
\end{align*}
Therefore, by Equation~(\ref{eq:hP_T}),
\begin{align*}
h(\Delta(\hat{P}_T),x) & = \sum_{S \subseteq T} 
\beta_{\hat{P}} (S) x^{|S|} \\ & =  
\sum_{k=0}^n h_k(P) \sum_{S \subseteq T} 
\# \{ w \in \fS_{n+1} : \, w(1) = k+1, \Des(w)
	 = n+1 - S \} \, x^{|n+1-S|} \\ & = 
\sum_{k=0}^n h_k(P) 
\sum_{w \in \fS_{n+1} : \, w(1) = k+1, \Des(w)
	    \subseteq n+1-T} x^{\des(w)} \\ & =
\sum_{k=0}^n h_k(P) p^{n+1-T}_{n,k}(x)
\end{align*}
and the proof follows from Theorem~\ref{thm:perms},
Lemma~\ref{lem:interlace-rec} (a) and the fact that 
$p^{n+1-T}_{n,0}(x) = A^{n-T}_n(x) = A^T_n(x)$.
\end{proof}

\noindent
\textbf{Colored permutations.} As an application, 
let us generalize part (a) of Theorem~\ref{thm:mainA} 
to $r$-colored permutations. An \emph{$r$-colored 
permutation} of the set $[n]$ is defined as a pair
$w \times \rzz$, where $w = (w(1), w(2),\dots,w(n)) 
\in \fS_n$ and $\rzz = (z_1, z_2,\dots,z_n) \in 
\{0, 1,\dots,r-1\}^n$. The number $z_i$ is thought 
of as the color assigned to $w(i)$. The set of all 
$r$-colored permutations of $[n]$ is denoted by 
$\fS_n[\ZZ_r]$. 

Let $u = w \times \rzz \in \fS_n[\ZZ_r]$ be an 
$r$-colored permutation, as before, and set $w(n+1) 
= n+1$ and $z_{n+1} = 0$. A \emph{descent} of $u$ is 
any index $i \in [n]$ such that either $z_i > z_{i+1}$, 
or $z_i = z_{i+1}$ and $w(i) > w(i+1)$. Thus, $n$ is 
a descent of $u$ if and only if $w(n)$ has nonzero 
color. As usual, we denote by $\Des(u)$ and $\des(u)$ 
the set and the number of descents of $u \in \fS_n
[\ZZ_r]$, respectively. The polynomial 
\begin{equation} \label{eq:AnrT-def}
A^T_{n,r}(x) = \sum_{u \in \fS_n[\ZZ_r] : 
              \, \Des(u) \subseteq T} x^{\des(u)},
\end{equation}
defined for $T \subseteq [n]$, provides a common 
generalization of $A^T_n(x)$ (the special case 
$r=1$) and the $r$-colored Eulerian polynomial 
$A_{n,r}(x)$ (the special case $T=[n]$), introduced
and studied by Steingr\'imsson \cite{Stei92, Stei94}. 
The latter was shown to be real-rooted in 
\cite[Theorem~3.19]{Stei92} 
\cite[Theorem~19]{Stei94}.
\begin{theorem} \label{thm:rperms} 
The polynomial $A^T_{n,r}(x)$ is real-rooted and 
interlaced by $A^T_n(x)$ for all positive integers 
$n, r$ and every $T \subseteq [n]$.
\end{theorem}

\begin{proof}
We will apply Theorem~\ref{thm:mainA} to the poset 
of $r$-colored subsets of the set $[n]$, defined as 
follows. We consider the subsets $\Omega$ of $[n] 
\times \{0, 1,\dots,r-1\}$ for which for every $i 
\in [n]$ there is at most one $j \in \{0, 
1,\dots,r-1\}$ such that $(i, j) \in \Omega$ and 
let $P$ be the set of all 
such subsets, partially ordered by inclusion. Thus, 
$P$ is a graded simplicial poset of rank $n$ which 
is isomorphic to the Boolean lattice $B_n$ for 
$r=1$. It was shown in the proof of 
\cite[Theorem~1.3]{Ath14} that $P$ is shellable, 
hence Cohen--Macaulay, and that $\beta_{\hat{P}}(S)$
is equal to the number of $r$-colored permutations
$u \in \fS_n[\ZZ_r]$ with descent set 
equal to $S$, for every $S \subseteq [n]$. As a 
result, in view of Equation~(\ref{eq:hP_T}),
\[ h(\Delta(\hat{P}_T),x) = \sum_{S \subseteq T} 
   \beta_{\hat{P}} (S) x^{|S|} = 
\sum_{u \in \fS_n[\ZZ_r] : \, \Des(u) \subseteq T} 
x^{\des(u)} = A^T_{n,r}(x) \]
for every $T \subseteq [n]$ and the proof follows 
from Theorem~\ref{thm:mainA}.
\end{proof}

\section{Noncrossing partition lattices}
\label{sec:noncrossing}

This section proves Theorem~\ref{thm:mainB}. 
We first recall the definition of $\NC_W$. Let $W$ 
be a finite Coxeter group with rank 
$r_W$ and set of reflections $\tT$. For $\alpha \in 
W$ we denote by $\ell_\tT(\alpha)$ the smallest $k$ 
such that $\alpha$ can be written as a  product of 
$k$ reflections in $\tT$. We define the partial order 
$\preceq$ on $W$ by letting $\alpha \preceq \beta$ 
if $\ell_\tT(\alpha) + \ell_\tT(\alpha^{-1} \beta) = 
\ell_\tT(\beta)$, in other words if there exists a 
shortest factorization of $\alpha$ into reflections 
which is a prefix of such a shortest factorization 
of $\beta$. Then, $\NC_W$ is defined as the closed 
interval $[e, \gamma]$ in $(W, \preceq)$, where $e 
\in W$ is the identity element and $\gamma$ is any 
Coxeter element of $W$. The noncrossing partition 
poset $\NC_W$ is a rank-symmetric, graded 
lattice with rank function $\ell_\tT$ and rank $r_W$; 
its combinatorial type is independent of the choice
of $\gamma$. A detailed exposition of noncrossing 
partition lattices can be found in 
\cite[Chapter~2]{Arm09}.

An important role in our first proof of 
Theorem~\ref{thm:mainB}~(a) will be played by the 
polynomials
\[ p_{n,k}(x) = \sum_{w \in \fS_{n+1} : \, w(1) = 
                k+1} x^{\des(w)}, \]
where $k \in \{0, 1,\dots,n\}$, discussed in  
Section~\ref{sec:perms}. 
\begin{proof}[First proof of 
Theorem~\ref{thm:mainB}~(a)] Let $r_W = n$ be the 
rank of $W$. We first assume that $W$ is 
irreducible. Then, the zeta polynomial of $\NC_W$
is given by the formula (see, for instance, 
\cite[Theorem~3.5.2]{Arm09})
\begin{align*} 
  \zZ(\NC_W, m) & = \frac{1}{|W|} \prod_{i=1}^n 
                   (mh + d_i) \\ & = 
	 \frac{1}{|W|} \prod_{i=1}^n \left( 
	 (h-d_i)m + d_i(m+1) \right), 
\end{align*} 
where $h$ is the Coxeter number of $W$ and $d_1, 
d_2,\dots,d_n$ are its degrees. Since $d_i \le h$ 
for every $i$, the second expression shows that 
$\zZ(\NC_W, m)$ can be written as a nonnegative 
linear combination of the polynomials $m^k
(1+m)^{n-k}$ for $k \in \{0, 1,\dots,n\}$. 
Moreover, this must be the case for every $W$, 
since $\NC_W$ is isomorphic to the product of
posets $\prod_{i=1}^\ell \NC_{W_i}$, where $W_1, 
W_2,\dots,W_\ell$ are the irreducible components 
of $W$, and hence $\zZ(\NC_W, m) = 
\prod_{i=1}^\ell \zZ(\NC_{W_i}, m)$. 

In view of Equations~(\ref{eq:zeta}) 
and~(\ref{eq:pnk-gen}), we conclude that 
$h(\Delta(\NC_W),x)$ can be written as a 
nonnegative linear combination of the polynomials 
$p_{n,k}(x)$ for $k \in \{0, 1,\dots,n\}$. Since 
$(p_{n,k}(x))_{0 \le k \le n}$ is an interlacing 
sequence, this and Lemma~\ref{lem:interlace-rec}
imply that $h(\Delta(\NC_W),x)$ is real-rooted 
and is interlaced by the Eulerian polynomial 
$A_n(x)$. 
\end{proof}

\begin{remark} \label{rem:theo} \rm
An explicit expression for $\zZ(\NC_W, m)$ as a 
nonnegative linear combination of the polynomials 
$m^k(1+m)^{n-k}$ for $k \in \{0, 1,\dots,n\}$ can
be deduced from results of \cite{BJV19, JV15} (see 
\cite[Section~4.6]{BJV19}). Specifically, assuming 
that $W$ is irreducible of rank $n$, we have 
\[ \zZ(\NC_W, m) = \sum_{k=0}^{n-1} 
   \frac{JV(W;k)}{n!} \cdot m^k (1+m)^{n-k}, \]
where $JV(W;k)$ is equal to the number of shortest
factorizations $\gamma = \tau_1 \tau_2 \cdots \tau_n$
of the Coxeter element $\gamma$ into reflections 
such that there are exactly $k$ indices $i \in [n-1]$
for which $\tau_1 \tau_2 \cdots \tau_i$ is greater 
than $\tau_1 \tau_2 \cdots \tau_{i+1}$ in the Bruhat
order on $W$.
\qed
\end{remark}

The second proof of part (a) and the proof of part 
(b) of Theorem~\ref{thm:mainB} are based on 
explicit combinatorial interpretations of the 
polynomial $h(\Delta(\NC_W),x)$ for the 
irreducible finite Coxeter groups of classical 
types. Before stating them, we need to
introduce some definitions and notation. A 
\emph{descent} of a word $w \in [r]^n$ is any 
index $i \in [n-1]$ such that $w(i) \ge w(i+1)$. 
We denote by $\dD_n$ the set of words $w \in \ZZ^n$ 
such that $(|w(1)|, w(2),\dots,w(n)) \in [n-1]^n$. 
A \emph{descent} of such a word $w \in \dD_n$ is 
defined as any index $i \in [n-1]$ such that 
\begin{itemize}
\itemsep=0pt
\item[$\bullet$]
$|w(i)| > w(i+1)$, or

\item[$\bullet$]
$w(i) = w(i+1) > 0$.
\end{itemize} 
As usual, we denote by $\Des(w)$ and $\des(w)$ the 
set and the number of descents, respectively, of a 
word $w$. 
\begin{proposition} \label{prop:noncrossing} 
Let $W$ be an irreducible finite Coxeter group of 
Coxeter type $\xX$. Then,

\[ h(\Delta(\NC_W),x) = \begin{cases}
    \displaystyle \frac{1}{n} \sum_{w \in [n]^{n-1}} 
		x^{\des(w)}, & \text{if $\xX = A_{n-1}$} \\ & \\
    \displaystyle \sum_{w \in [n]^n} x^{\des(w)}, 
		& \text{if $\xX = B_n$} \\
    & \\ \displaystyle \sum_{w \in \dD_n} x^{\des(w)},
    & \text{if $\xX = D_n$.}  \end{cases} \]
Moreover,
\[  h(\Delta(\NC_W),x) = \begin{cases}
    1 + (m-1)x, & \text{if $\xX = I_2(m)$} \\
    1 + 28x + 21x^2, & \text{if $\xX = H_3$} \\
    1 + 275x + 842x^2 + 232x^3, & \text{if $\xX = H_4$} \\
    1 + 100x + 265x^2 + 66x^3, & \text{if $\xX = F_4$} \\
    1 + 826x + 10778x^2 + 21308x^3 + 8141x^4 + 418x^5, 
		               & \text{if $\xX = E_6$} \\
    1 + 4152x + 110958x^2 + 446776x^3 + 412764x^4 & \\
		\ \ + \ 85800x^5 + 2431x^6, & \text{if $\xX = E_7$} \\
    1 + 25071x + 1295238x^2 + 9523785x^3 + 17304775x^4 & \\
		\ \ +  8733249x^5 + 1069289x^6 + 17342x^7, & 
		\text{if $\xX = E_8$}.  
		\end{cases} \]
\end{proposition}

\begin{proof}
Let us write $P = \NC_W$ and first consider the 
case $\xX = A_{n-1}$. The explicit formula of 
\cite[Theorem~3.2]{Ed80} (see also 
\cite[p.~196]{Rei97}) for the entries of the flag 
$f$-vector of $P$ can be rewritten as 
\[ \alpha_P(T) = \frac{1}{n} \, \# \{ w \in 
   [n]^{n-1} : \Des(w) \subseteq T \} \]
for $T \subseteq [n-2]$. From Equation~(\ref{eq:ba})
it readily follows that 
\[ \beta_P(S) = \frac{1}{n} \, \# \{ w \in 
   [n]^{n-1} : \Des(w) = S \} \]
for $S \subseteq [n-2]$ and hence that 
\[ h(\Delta(P),x) = \sum_{S \subseteq [n-2]} 
   \beta_P (S) x^{|S|} = \frac{1}{n} 
	 \sum_{w \in [n]^{n-1}} x^{\des(w)}. \]
One can reach the same conclusion by using the 
combinatorial interpretation of $\beta_P(S)$ in 
terms of parking functions, given in 
\cite[Proposition~3.2]{Sta97}. The proof of the 
formula for $\xX = B_n$ is entirely similar, 
once one rewrites the formula of 
\cite[Proposition~7]{Sta97} for the flag 
$f$-vector of $P$ as 
\[ \alpha_P(T) = \# \{ w \in [n]^n : \Des(w) 
   \subseteq T \} \]
for $T \subseteq [n-1]$. 

Let us now consider the case $\xX = D_n$, which 
is more involved: it is not true any more that
$\beta_P(S)$ is equal to the number of words $w 
\in \dD_n$ with descent set equal to $S$. Let us
write $\bar{P} = P \sm \{\hat{0}, \hat{1}\}$.
The formula of \cite[Theorem~1.2]{AR04} for the
flag $f$-vector of $P$ shows that
\begin{align*} f_{k-1}(\Delta(\bar{P})) & = 2 
   \sum_{(a_1, a_2,\dots,a_{k+1}) \vDash n} 
	 {n-1 \choose a_1} {n-1 \choose a_2} \cdots 
	 {n-1 \choose a_{k+1}} \, + \\ & 
	 \sum_{(a_1, a_2,\dots,a_{k+1}) \vDash n} 
	 \sum_{i=1}^{k+1}
	 {n-1 \choose a_1} \cdots {n-2 \choose a_i - 2} 
	 \cdots {n-1 \choose a_{k+1}}, \end{align*}
where the first two sums run through all 
compositions $(a_1, a_2\dots,a_{k+1})$ of $n$ 
with $k+1$ parts. Using the fact that 
${n-2 \choose a_i - 2} = \frac{a_i - 1}{n-1} 
{n-1 \choose a_i - 1}$, changing the order of 
summation in the double sum and replacing $a_i$ 
with $a_i + 1$ yields that 

\begin{align*} f_{k-1}(\Delta(\bar{P})) & = 2 
   \sum_{(a_1, a_2\dots,a_{k+1}) \vDash n} 
	 {n-1 \choose a_1} {n-1 \choose a_2} \cdots 
	 {n-1 \choose a_{k+1}} \, + \\ & \sum_{i=1}^{k+1}
	 \sum_{(a_1, a_2,\dots,a_{k+1}) \vDash n-1} 
	 \frac{a_i}{n-1}
	 {n-1 \choose a_1} \cdots {n-1 \choose a_2} 
	 \cdots {n-1 \choose a_{k+1}}. \end{align*}
Changing again the order of summation in the 
double sum, since $\sum_{i=1}^{k+1} a_i = n-1$, 
we find that
	
\begin{align*} f_{k-1}(\Delta(\bar{P})) & = 2
	 \sum_{(a_1, a_2,\dots,a_{k+1}) \vDash n} 
	 {n-1 \choose a_1} {n-1 \choose a_2} \cdots 
	 {n-1 \choose a_{k+1}} \, + \\ & 
	 \sum_{(a_1, a_2,\dots,a_{k+1}) \vDash n-1} 
	 {n-1 \choose a_1} \cdots {n-1 \choose a_2} 
	 \cdots {n-1 \choose a_{k+1}}. \end{align*}
We may rewrite this formula as  

\begin{align*} f_{k-1}(\Delta(\bar{P})) & = 2
	 \sum_{T \subseteq [n-1], \, |T|=k} 
	 \# \{ w \in [n-1]^n : \Des(w) \subseteq T\} 
	 \, + \\ & 
	 \sum_{T \subseteq [n-2], \, |T|=k} 
	 \# \{ w \in [n-1]^{n-1} : \Des(w) \subseteq T\}, 
\end{align*}
whence

\begin{align*} h(\Delta(P),x) & = h(\Delta(\bar{P}),x) 
  = \sum_{k=0}^{n-1} f_{k-1}(\Delta(\bar{P})) \, 
	x^k (1-x)^{n-1-k} \\ & = 2 \, \sum_{k=0}^{n-1}
	\sum_{T \subseteq [n-1], \, |T|=k} 
	\# \{ w \in [n-1]^n : \Des(w) \subseteq T\} x^k 
	(1-x)^{n-1-k} \\ & + \sum_{k=0}^{n-1}
	\sum_{T \subseteq [n-2], \, |T|=k} 
	\# \{ w \in [n-1]^{n-1} : \Des(w) \subseteq T\} 
	x^k (1-x)^{n-1-k} \\ & = 
	2 \sum_{T \subseteq [n-1]} \# \{ w \in [n-1]^n : 
	\Des(w) \subseteq T\} x^{|T|} (1-x)^{n-1-|T|}
	\\ & + \sum_{T \subseteq [n-2]} 
	\# \{ w \in [n-1]^{n-1} : \Des(w) \subseteq T\} 
	x^{|T|} (1-x)^{n-1-|T|}.
\end{align*}
Setting $\Des(w) = S$ in each sum, summing over 
all $S \subseteq T$ and changing the order of 
summation yields that

\begin{align*} h(\Delta(P),x) & = 
	2 \sum_{S \subseteq [n-1]} \# \{ w \in [n-1]^n : 
	\Des(w)=S \} \sum_{S \subseteq T \subseteq [n-1]} 
	x^{|T|} (1-x)^{n-1-|T|} \\ & + 
	\sum_{S \subseteq [n-2]} \# \{ w \in [n-1]^{n-1} : 
	\Des(w)=S \} \sum_{S \subseteq T \subseteq [n-2]} 
	x^{|T|} (1-x)^{n-1-|T|}
\end{align*}
and hence that 

\begin{align} h(\Delta(P),x) & = 
	2 \sum_{S \subseteq [n-1]} \# \{ w \in [n-1]^n : 
	\Des(w)=S \} x^{|S|} \nonumber \\ & + 
	(1-x) \sum_{S \subseteq [n-2]} \# \{ w \in 
	[n-1]^{n-1} : \Des(w)=S \} x^{|S|} \nonumber \\
	& = 2 \sum_{w \in [n-1]^n} x^{\des(w)} + (1-x) 
	\sum_{w \in [n-1]^{n-1}} x^{\des(w)}. 
	\label{eq:Dcomputation} 
\end{align}
Considering the cases $|w(1)| \ne w(2)$ and $|w(1)| 
= w(2)$ for a word $w \in \dD_n$ shows that the 
number of words $w \in \dD_n$ with $\des(w)=k$ is 
equal to the coefficient of $x^k$ in the 
expression~(\ref{eq:Dcomputation}) and the proof 
follows.

The exceptional types are handled via
straightforward computations via Sage~\cite{Sage}.
\end{proof}

\begin{proof}[Second proof of 
Theorem~\ref{thm:mainB}~(a)]
We recall that $\NC_W$ is isomorphic to the product 
of posets $\prod_{i=1}^\ell \NC_{W_i}$, where $W_1, 
W_2,\dots,W_\ell$ are the irreducible components 
of $W$. This fact and Lemma~\ref{lem:product} show
that we may assume that $W$ is irreducible.

Let us first consider the case of groups of type $D$.
By Proposition~\ref{prop:noncrossing}, it suffices 
to show that $h_n(x) := \sum_{w \in \dD_n} 
x^{\des(w)}$ is real-rooted for every $n \ge 2$. 
For $k \ge 2$, we denote by $\dD_{n,k}$ the set of 
words $w \in \ZZ^k$ such that $(|w(1)|, 
w(2),\dots,w(k)) \in [n-1]^k$ and note that $\dD_{n,n} 
= \dD_n$. We define the notion of descent for words 
$w \in \dD_{n,k}$ just as in the special case $k=n$ 
and set 
\[ h_{n,k,j}(x) = \sum_{w \in \dD_{n,k} : \, 
   w(k) = j} x^{\des(w)} \]
for $j \in [n-1]$. We will prove that $(h_{n,k,n-1}
(x),\dots,h_{n,k,2}(x), h_{n,k,1}(x))$ is an 
interlacing sequence of real-rooted polynomials for 
all $n,k \ge 2$ by induction on $k$. This holds for 
$k=2$ since then $h_{n,k,j}(x) = (2j-1) + (2n-2j-1)x$ 
for every $j \in [n-1]$. The inductive step follows 
by an application of part (c) of 
Lemma~\ref{lem:interlace-rec}, since 
\[ h_{n,k+1,j}(x) = \sum_{i=1}^{j-1} h_{n,k,i}(x) \, 
   + \, x \sum_{i=j}^{n-1} h_{n,k,i}(x) \]
for $j \in [n-1]$. In particular, $h_{n,n+1,1}(x) = 
xh_n(x)$ is real-rooted for every $n \ge 2$ and 
hence so is $h_n(x)$. 

A similar (and even simpler) 
argument shows that $\sum_{w \in [r]^n} x^{\des(w)}$
is real-rooted for all $n,r \ge 1$. This covers the 
cases of groups of types $A$ and $B$.
The exceptional groups can be treated with a case by 
case verification. 
\end{proof}

\noindent
\textbf{Symmetric decompositions.}
The second part of Theorem~\ref{thm:mainB} will 
be proven by an application a result of 
Jochemko~\cite{Jo21}, after the expressions of 
Proposition~\ref{prop:noncrossing} for
$h(\Delta(\NC_W),x)$ are suitably rewritten.
For a polynomial or formal power series $H(x) = 
\sum_{n \ge 0} h_n x^n \in \CC[[x]]$ we use the 
notation $\sS_r (H(x)) = \sum_{n \ge 0} h_{rn}x^n$.

\begin{lemma} \label{lem:word-enu} 
Let $E_{n,r}(x) = \sum_{w \in [r]^n} x^{\des(w)}$. 
Then, 
\[ x^n E_{n,r}(1/x) = \sS_r \left( x (1 + x + x^2 + 
   \cdots + x^{r-1})^{n+1} \right) \]
for all $n,r \ge 1$. 
\end{lemma}

\begin{proof}
First we relate the polynomials $E_{n,r}(x)$ to
the
\[ \tilde{E}_{n,r}(x) := \sum_{w \in \wW_{n,r}} 
   x^{\asc^*(w)}, \]
where $\wW_{n,r}$ is the set of words $w: \{0, 
1,\dots,n\} \to [r]$ with $w(0) = 1$ and $\asc^*(w)$
is the number of indices $i \in [n]$ such that 
$w(i-1) < w(i)$. We note that 
\[ x^{n-1} E_{n,r}(1/x) = \sum_{w \in [r]^n} 
   x^{n-1-\des(w)} = \sum_{w \in [r]^n} 
	 x^{\asc(w)}, \]
where $\asc(w) = n-1-\des(w)$ is the number of 
strict ascents of $w \in [r]^n$. Distinguishing the
cases $w(1) = 1$ and $w(1) \ge 2$ for such a word 
and for a word $w \in \wW_{n,r}$ we get 

\begin{align*}
x^{n-1} E_{n,r}(1/x) & = \tilde{E}_{n-1,r}(x) + 
  \sum_{w \in [r]^n : \, w(1) \ge 2} 
x^{\asc(w)}, \\ \tilde{E}_{n,r}(x) & = 
  \tilde{E}_{n-1,r}(x) + x
	\sum_{w \in [r]^n : \, w(1) \ge 2} x^{\asc(w)}.
\end{align*}
These equalities imply that 
\begin{equation} \label{eq:Erelation}
x^n E_{n,r}(1/x) = \tilde{E}_{n,r}(x) + (x-1)
  \tilde{E}_{n-1,r}(x).
\end{equation}

We now recall that
\begin{equation} \label{eq:Eold}
\tilde{E}_{n,r}(x) = \sS_r \left( (1 + x + x^2 + 
  \cdots + x^{r-1})^{n+1} \right).
\end{equation}
This formula follows from the identity
\begin{equation} \label{eq:savage}
\sum_{m \ge 0} {n+rm \choose n} x^m =
\frac{\tilde{E}_{n,r}(x)} {(1-x)^{n+1}},
\end{equation}
which can be proved by a standard `placing balls 
into boxes' argument  (see \cite[Corollary~8]{SS12} 
for a $q$-analogue) and the computation

\begin{align*}
\sum_{m \ge 0} {n+rm \choose n} x^m & = \sS_r \left( 
\frac{1}{(1-x)^{n+1}} \right) = \sS_r
\left( \frac{(1 + x + x^2 + \cdots + x^{r-1})^{n+1}}
{(1-x^r)^{n+1}} \right) \\
& = \frac{\sS_r \left( (1 + x + x^2 + \cdots + 
x^{r-1})^{n+1} \right)}{(1-x)^{n+1}}.
\end{align*}
Combining Equations~(\ref{eq:Erelation}) 
and~(\ref{eq:Eold}) we get 

\begin{align*}
x^n E_{n,r}(1/x) & = \sS_r \left( (1 + x + x^2 + 
  \cdots + x^{r-1})^{n+1} \right) 
  + (x-1) \sS_r \left( (1 + x + x^2 + \cdots + 
	x^{r-1})^n \right) \\
& = \sS_r \left( (1 + x + x^2 + \cdots + 
  x^{r-1})^{n+1} + (x^r-1) (1 + x + x^2 + \cdots + 
	x^{r-1})^n \right) \\ & = \sS_r
	\left( x (1 + x + x^2 + \cdots + x^{r-1})^{n+1}  
  \right)
\end{align*}
and the proof follows.
\end{proof}

The following result of Jochemko~\cite{Jo21} will
be applied in the proof of 
Theorem~\ref{thm:mainB}~(b).

\begin{theorem} \label{thm:Jo} 
{\rm (\cite[Theorem~1.1]{Jo21})}
Let $h(x) = h_0 + h_1x + \cdots + h_d x^d$ be a 
polynomial of degree $s \le d$ with nonnegative 
coefficients such that 
\begin{itemize}
\itemsep=0pt
\item[$\bullet$]
$h_0 + h_1 + \cdots + h_i \ge h_d + h_{d-1} + 
\cdots + h_{d-i+1}$, and

\item[$\bullet$]
$h_0 + h_1 + \cdots + h_i \le h_s + h_{s-1} + 
\cdots + h_{s-i}$
\end{itemize} 
for all $i$. Then, $\sS_r \left( h(x) (1 + x + x^2 
+ \cdots + x^{r-1})^{d+1} \right)$ has a nonnegative 
real-rooted symmetric decomposition with respect to 
$d$ whenever $r \ge \max \{s, d+1-s\}$.
\end{theorem}

\begin{proof}[Proof of Theorem~\ref{thm:mainB}~(b)]
Using the notation of Lemma~\ref{lem:word-enu}, by
Proposition~\ref{prop:noncrossing} and its proof we 
have

\[ h(\Delta(\NC_W),x) = \begin{cases} \displaystyle
    (1/n) E_{n-1,n}(x), & \text{if $\xX = A_{n-1}$} 
		\\ \displaystyle E_{n,n}(x), & 
		\text{if $\xX = B_n$} \\ \displaystyle
    2 E_{n,n-1}(x) + (1-x) E_{n-1,n-1}(x),
    & \text{if $\xX = D_n$.}  \end{cases} \]
In view of Lemma~\ref{lem:word-enu}, these formulas 
may be rewritten as

\[ x^{r_W} h(\Delta(\NC_W),1/x) = \begin{cases} 
   \displaystyle (1/n) \sS_n \left( x (1 + x + x^2 + 
	 \cdots + x^{n-1})^n \right), 
	 & \text{if $\xX = A_{n-1}$} \\ \displaystyle 
	 \sS_n \left( x (1 + x + x^2 + \cdots + 
	 x^{n-1})^{n+1} \right), & 
	 \text{if $\xX = B_n$} \\ \displaystyle \sS_{n-1}
   \left( (x+x^2) (1 + x + x^2 + \cdots + 
	 x^{n-2})^{n+1} \right),
   & \text{if $\xX = D_n$.}  \end{cases} \]
These expressions and Theorem~\ref{thm:Jo} imply 
in each case that $x^{r_W} h(\Delta(\NC_W),1/x)$ 
has a nonnegative
real-rooted symmetric decomposition with 
respect to $r_W$. Since $h(\Delta(\NC_W),x)$ has
degree $r_W - 1$, it has a nonnegative 
real-rooted symmetric decomposition with respect 
to $r_W - 1$. The exceptional groups are again 
handled by a routine case by case verification.
\end{proof}

We close this section with the analogue of 
Question~\ref{que:convex-ear}.
\begin{question} \label{que:convex-ear2} 
Let $h(\Delta(\NC_W),x) = \sum_{i=0}^r h_i(W) x^i$, 
where $r = r_W - 1$.
\begin{itemize}
\itemsep=0pt
\item[(a)]
Does the order complex $\Delta(\overline{\NC}_W)$ 
of the noncrossing partition lattice $\NC_W$ (with 
its minimum and maximum elements removed) have a 
convex ear decomposition? 

\item[(b)]
Do the inequalities
\[ \frac{h_0(W)}{h_r(W)} \le 
   \frac{h_1(W)}{h_{r-1}(W)} \le \cdots \le 
   \frac{h_r(W)}{h_0(W)} \]
hold?

\end{itemize}
\end{question}

\medskip
\noindent \textbf{Acknowledgments}. Part of the 
motivation behind Theorem~\ref{thm:mainA} was 
developed during the workshop `Interactions 
between Topological Combinatorics and 
Combinatorial Commutative Algebra', held at BIRS 
(Banff, Canada) in April 2023. The first named 
author wishes to thank the organizers Mina Bigdeli, 
Sara Faridi, Satoshi Murai and Adam Van Tuyl for 
the invitation and the participants for useful 
discussions. The authors also wish to thank 
Christian Stump for help with the computation of 
the chain polynomial of the noncrossing partition
lattice of type $E_8$.

\end{document}